\documentclass[12pt,a4paper]{article}

\usepackage{amsmath,amssymb,amsthm,latexsym,color,times,mathrsfs}
\usepackage{tikz,fullpage}
\usepackage{xcolor,colortbl}
\usetikzlibrary{mindmap}
\usepackage[bookmarks=false,pdfstartview=FitH,
hyperfigures=true,colorlinks=true,citecolor=blue,
linkcolor=blue]{hyperref}
\usetikzlibrary{shapes,arrows,shadows,snakes}%
\usetikzlibrary{calc,positioning}%
\usepackage{pagecolor}

\def\pf{\noindent {\em Proof.}\ }
\def\qed{{\quad\rule{1mm}{3mm}\,}}
\def\JS{\mathscr{J\!\!S}}
\def\WD{\mathscr{W\!\!D}}
\def\le{\leqslant}
\def\ge{\geqslant}
\newtheorem{thm}{Theorem}
\newtheorem{lmm}{Lemma}

\newtheorem{pro}{Proposition}
\theoremstyle{remark}\newtheorem{Rem}{Remark}
\theoremstyle{remark}\newtheorem{Def}{Definition}

\definecolor{brown}{rgb}{0.5,0.25,0}
\definecolor{dg}{rgb}{0,0.3,0}
\definecolor{gold}{rgb}{1,0.84,0.}

\def\CTLS{\textbf{\texttt{CTLS}}}
\def\MGLS{\textbf{\texttt{MGLS}}}
\def\RPSLS{\textbf{\texttt{RPSLS}}}
\def\GJLS{\textbf{\texttt{GJLS}}}

\title{From coin-tossing to rock-paper-scissors and beyond:\\
       A log-exp gap theorem for selecting a leader}
\author{Michael FUCHS\\
    Department of Applied Mathematics \\
    National Chiao Tung University \\
    Hsinchu 300 \\ Taiwan
\and Hsien-Kuei HWANG \\
    Institute of Statistical Science \\
    Academia Sinica\\
    Taipei 115\\ Taiwan
\and Yoshiaki ITOH \\
    Institute of Statistical Mathematics \\
    10-3 Midori-cho, Tachikawa \\
    Tokyo 190-8562 \\ Japan
}

\begin{document}
\maketitle
\pagecolor{yellow!10}

{\renewcommand{\thefootnote}{}
\footnotetext{\hspace*{-0.6cm}
}}

\begin{abstract}

A class of games for finding a leader among a group of candidates is
studied in detail. This class covers games based on coin-tossing and rock-paper-scissors as special cases
and its complexity exhibits similar stochastic behaviors: either of logarithmic mean and
bounded variance or of exponential mean and exponential variance. Many applications
are also discussed.

\end{abstract}

\section{Introduction}\label{intro}

Selecting a leader or a representative by fair random mechanisms
without relying on \emph{a priori} information of the candidates has
long been used in diverse contexts and civilizations. Typical
examples range from sortition (or allotment) in the West and
rock-paper-scissors (or Janken) in the East; see Wikipedia's pages on
sortition and rock-paper-scissors for more information. This paper is
concerned with the analysis of a class of \emph{leader selection
algorithms} (or \emph{leader election algorithms}) that are used to
select a leader among a group of $n$ candidates. These algorithms
have widespread applications in diverse areas; see Figure \ref{fig-1}
for a summary and below for more descriptions.

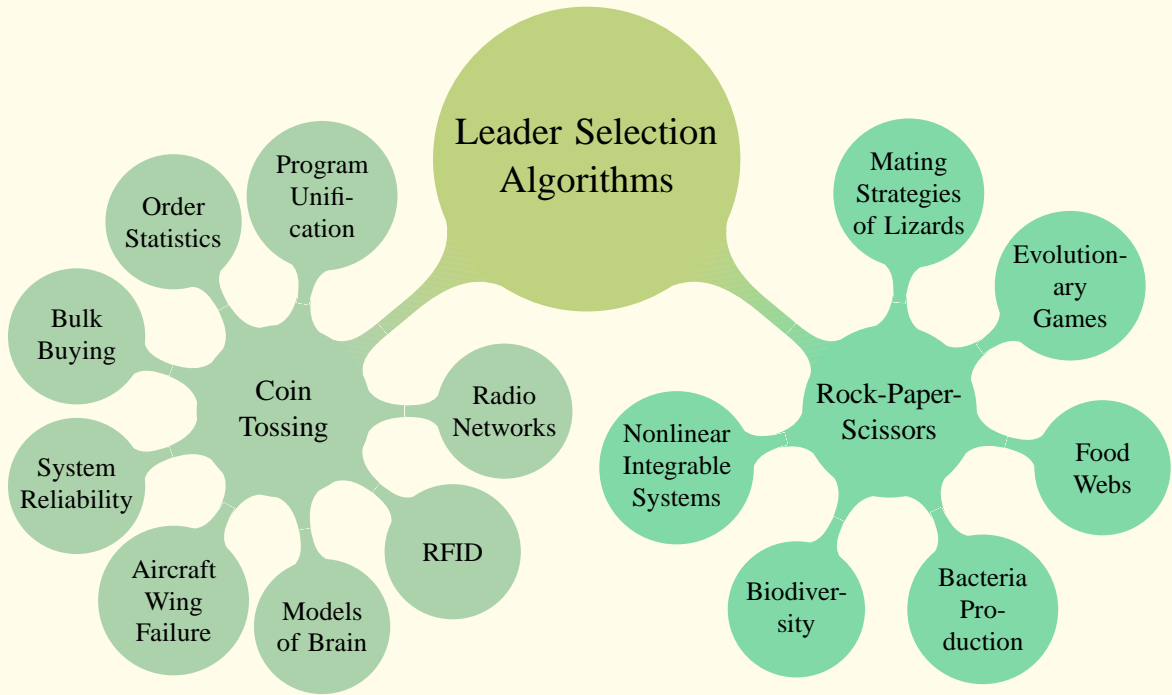
\begin{figure}
\begin{center}
\begin{tikzpicture}
\path[mindmap, concept color=dg!50!yellow!50, text=black,
level 1 concept/.append style={level distance=52mm,
sibling angle=100}, level 2 concept/.append
style={sibling angle=40}]
node[concept] {Leader Selection Algorithms}
[clockwise from=-40]
child[concept color=blue!30!green!50, text=black] {
  node[concept] {Rock-Paper-Scissors}
  [clockwise from=85,
  level 2 concept/.append style={sibling angle=50}]
  child {node[concept] {Mating Strategies of Lizards}}
  child {node[concept] {Evolution- \\ary Games}}
  child {node[concept] {Food Webs}}
  child {node[concept] {Bacteria Production}}
  child {node[concept] {Biodiver- \\sity}}
  child {node[concept] {Nonlinear Integrable Systems}}}
child[concept color=green!30!gray!50, text=black] {
  node[concept] {Coin Tossing}
  [clockwise from=0]
  child{node[concept] {Radio Networks}}
  child{node[concept] {RFID}}
  child{node[concept] {Models of Brain}}
  child{node[concept] {Aircraft Wing Failure}}
  child{node[concept] {System Reliability}}
  child{node[concept] {Bulk Buying}}
  child{node[concept] {Order Statistics}}
  child{node[concept] {Program Unification}}};
\end{tikzpicture}
\end{center}
\caption{Leader selection procedures and
applications.}\label{fig-1}
\end{figure}

One easy and efficient way to solve the leader selection problem is
to use coin-tossing. A simple such procedure is described as follows.
Assume that we have a (possibly biased) coin with two outcomes
``head" and ``tail". Each of the $n$ candidates tosses an independent
coin and those who toss head go on to the next round. In case nobody
tosses head, the round is repeated with the same candidates. The
procedure ends when only one candidate is left who is then declared
the leader. We refer to this simple scheme as \CTLS\ (Coin-Tossing
Leader Selection). It has been used in many applications, for
instance, in the CTM Tree Protocol (after Capetanakis, Tsybakov and
Mikhailov), which is used to determine the order in which $n$
processors sharing a common communication channel send their
messages; see \cite{Ca} and \cite{TsMi}. More recent applications are
in ad-hoc radio networks (see Chapter 9 of \cite{NaOl}) and RFID
systems (see, e.g., \cite{HuWo}).

Many variants of \CTLS\ have been proposed and extensively studied.
One such variant consists of waiting until everyone is eliminated and
then declaring those who stayed the longest in the game as the
leaders. This variant, referred to as \MGLS\ (Maximum Geometric
Leader Selection), amounts to finding the maximum of $n$ i.i.d.\
geometric distributed random variables. The latter problem has a
longer history than the study of \CTLS, and the earliest publication
we found dated back to the early 1950s. Closely connected
applications include mathematical models of the brain \cite{voNe},
study of aircraft wing fatigue failure \cite{Lu} and system
reliability in general \cite{We}, order statistics of geometric
random variables \cite{MaWi}, skip-lists \cite{Devroye92}, bulk
buying of possibly defective items \cite{Sp,Ke} and program
unification techniques in concurrency enhancement methods
\cite{ReMa,ReSz}; see also the more recent study \cite{Ei} and the
book on bioinformatics \cite{EwGr} for further applications.

The complexity of \CTLS\ seemed to have been first analyzed by Bruss
and O'Cinneide \cite{BrOc} in the early 1990s, where they attributed
the problem of analyzing the number of rounds to identify a leader to
Bajaj and Mendieta \cite{BaMa}. In fact, in \cite{BrOc}, the authors
relied on the above connection to geometric random variables as an
approximation to \CTLS, and gave an analysis of \MGLS\ as well. A
more detailed analysis of \CTLS\ was given independently by Prodinger
\cite{Pr}. In particular, he showed that the expected number of
rounds, denoted by $\mathbb{E}(X_n)$, used by \CTLS\ to identify a
leader among $n$ candidates satisfies
\begin{align}\label{EXn-logn}
    \mathbb{E}(X_n)
	=\log_2n+\tfrac{1}{2}
	+P_0(\log_2n)+O\left(n^{-1}\right),
\end{align}
when all coins tossed are unbiased. Here $P_0(u)$ is a bounded
periodic function of period $1$ with amplitude less than $1.93\times
10^{-5}$. Such a minute yet nonzero periodic oscillation is a
characteristic feature for problems of a similar type, and tools for
deriving the corresponding Fourier expansions have been the major
focus of many papers. A similar periodic phenomena is also present in
the variance $\mathbb{V}(X_n)$, which is asymptotic to another
periodic function and thus of constant order (again in the unbiased
case). This together with the following limit distribution result
\begin{equation}\label{unbiased-ll}
    \mathbb{P}(X_n\le\lfloor\log_2n\rfloor+\ell)
    =\frac{2^{\{\log_2 n\}-\ell}}
    {\exp(2^{\{\log_2 n\}-\ell})-1}+O\left(n^{-1}\right),
\end{equation}
where $\{x\}$ denotes the fractional part of $x$, was derived in
\cite{FiMaSz}. Moreover, extensions of the above results to biased
coins were also considered; see \cite{JaSz}.

The logarithmic order \eqref{EXn-logn} shows that \CTLS\ is not only
simple but also very efficient in terms of cost complexity. We will
establish a more general asymptotic pattern of this type and clarify
when a leader selection procedure is more efficient than the others.

In addition to coin-tossing with binary outcomes, the natural idea of
allowing $m$-ary outcomes ($m\ge3$) with cyclic dominance has also
proved fruitful in diverse applications including leader selection.
The simplest such procedure is the ``rock-paper-scissors game" (RPS),
which is popular in many countries, notably in Japan, where it is
called ``Janken game", meaning the play between two fists. Many
annual tournaments and championships of such games on regional and
international level are held, and have received widespread media
attention. RPS originated in China and was imported to Japan about
a thousand years ago (with different rules).
% It was called ``bug-ken'' in which snake
% dominates frog, slug dominates snake, and frog dominates slug.
% How it was played is not clearly known.
% Three centuries ago in Japan,
% number-ken (without cyclic dominance), originally also
% played in China, was first changed to a children's game
% (with cyclic dominance) of village headman, fox, and gun.
Then it was modified to the current form of RPS about a century ago,
and its use spread later across Asia and to the western and the whole
world. See \cite{Li,Li95,OKSY98} for a detailed account. The rules
underlying RPS can be visualized by the following digraph, which
indicates the dominance relations:

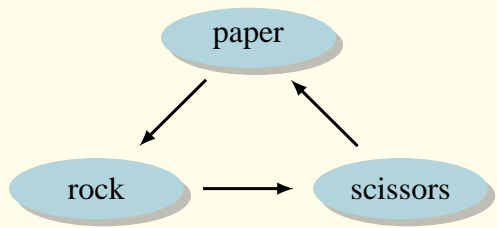
\begin{center}
\begin{tikzpicture}[oval2/.style={ellipse, fill=blue!55!green!30,
rounded corners, ,minimum width=2.3cm, drop shadow}]
\draw (0cm,0cm) node[oval2] (1) {rock};
\draw (2cm,2cm) node[oval2] (2) {paper};
\draw (4cm,0cm) node[oval2] (3) {scissors};
\draw[shorten <=7pt,shorten >=7pt,-latex,line width=0.4mm]
(1) -- (3);
\draw[shorten <=7pt,shorten >=7pt,-latex,line width=0.4mm]
(3) -- (2);
\draw[shorten <=7pt,shorten >=7pt,-latex,line width=0.4mm]
(2) -- (1);
\end{tikzpicture}
\end{center}

Many variants of RPS exist. Examples include the
rock-paper-scissors-well game (popular in Germany), and the
rock-paper-scissors-Spock-lizard game; see the figure below and
Wikipedia's page on RPS for more information.

\begin{figure}
\begin{center}
\begin{tikzpicture}
\node[] (left) at (0,0.1){
\begin{tikzpicture}[scale=1.3]
\draw (-8.21cm,3.51cm) node (11) {R};
\draw (-8.21cm,0.8cm) node (12) {P};
\draw (-5.5cm,0.8cm) node (13) {S};
\draw (-5.5cm,3.51cm) node (14) {W};

\draw[-latex,line width=0.4mm] (11) -- (13);
\draw[-latex,line width=0.4mm] (12) -- (11);
\draw[-latex,line width=0.4mm] (12) -- (14);
\draw[-latex,line width=0.4mm] (13) -- (12);
\draw[-latex,line width=0.4mm] (14) -- (11);
\draw[-latex,line width=0.4mm] (14) -- (13);
\end{tikzpicture}
};

\node[] (right) at (6,0.2){
\begin{tikzpicture}
\draw (0cm,4.31cm) node (1) {R};
\draw (-2.27cm,2.66cm) node (2) {P};
\draw (-1.4cm,0cm) node (3) {Sc};
\draw (1.4cm, 0cm) node (4) {Sp};
\draw (2.27cm,2.66cm) node (5) {L};

\draw[-latex,line width=0.4mm] (1) -- (3);
\draw[-latex,line width=0.4mm] (1) -- (5);
\draw[-latex,line width=0.4mm] (2) -- (1);
\draw[-latex,line width=0.4mm] (2) -- (4);
\draw[-latex,line width=0.4mm] (3) -- (2);
\draw[-latex,line width=0.4mm] (3) -- (5);
\draw[-latex,line width=0.4mm] (4) -- (1);
\draw[-latex,line width=0.4mm] (4) -- (3);
\draw[-latex,line width=0.4mm] (5) -- (2);
\draw[-latex,line width=0.4mm] (5) -- (4);
\end{tikzpicture}
};

\node[yshift=-2.75cm] (L1cap) at (left)
{\small Rock, Paper, Scissors, Well};
\node[yshift=-0.5cm] (L2cap) at (L1cap) {\small (Germany)};
\node[xshift=6cm] (R1cap) at (L1cap)
{\small Rock, Paper, Scissors,};
\node[yshift=-0.5cm] (R2cap) at (R1cap) {\small Spock, Lizard};
\end{tikzpicture}
\end{center}
\end{figure}

The usefulness of the Janken game and its variants is not limited to
select a winner or a loser, but also to broad applications in other
areas. Examples are found in evolutionary game theory (see
\cite{HoSi,Sm}) and in biology (the field data on alternative male
strategies of side-blotched-lizards \cite{SiLi} being a well-known
example). Other biological uses are found in food-webs, antibiotic
production of bacteria, and biodiversity. In physical applications,
Janken games were also encountered in interacting particle systems
with cyclic dominance \cite{It1} which have a Lotka-Volterra system
as a deterministic approximation. These systems can be extended to
nonlinear integrable systems; see \cite{Bo,It2,It3}. The introduction
of the spatial structure as lattice Lotka-Volterra system drastically
enriches the dynamics of interacting particles systems and yields
interesting simulation results; see \cite{FrKr,KnKrWeFr,SzFa,Ta}.

When used to select a leader among $n$ candidates, it turns out that
RPS is very inefficient in that it requires an \emph{exponential}
number of rounds to resolve the overwhelming ties, in contrast to the
logarithmic complexity \eqref{EXn-logn} for \CTLS. More precisely,
the procedure follows along the natural way and if only two different
hands are present, then the group whose hand dominates that of the
other will go on to the next stage; otherwise, the game is in a tie
and has to be repeated. The game ends when only one candidate is left
who is then the leader. We will call this procedure \RPSLS.

Maehara and Ueda \cite{MaUe} proved that when $p_1=p_2=p_3=\frac13$,
the number of rounds, say $Y_n$, used by \RPSLS\ satisfies
\[
    \mathbb{E}(Y_n)
	\sim\tfrac13\left(\tfrac32\right)^n,
\]
(cf. \eqref{EXn-logn}) and, furthermore,
\[
    \frac{Y_n}{\frac13\left(\frac32\right)^n}
	\stackrel{d}{\longrightarrow}\text{Exp}(1),
\]
where $\stackrel{d}{\longrightarrow}$ denotes convergence in
distribution and $\text{Exp}(1)$ represents an exponential
random variable with mean $1$. While the expected
exponential complexity of \RPSLS\ under uniform distribution
of the hands cannot compete with the expected logarithmic
one of \CTLS, for very small $n$, we have the following
numerical values:

\begin{center}
    \begin{tabular}{|c||c|c|c|c|c|c|c|c|}\hline
        Scheme &$n$ & $2$ & $3$ & $4$ & $5$ & $6$
        & $7$ & $8$ \\ \hline\hline
        \CTLS &$\mathbb{E}(X_n)$ &
        $4$ & $4.83$ & $5.52$ & $6.09$
		& $6.58$
        & \cellcolor{blue!15}$6.99$
		& \cellcolor{blue!15}$7.35$ \\ \hline
        \RPSLS & $\mathbb{E}(Y_n)$ & \cellcolor{blue!15}$1.5$
		& \cellcolor{blue!15}$2.25$
		& \cellcolor{blue!15}$3.21$
		& \cellcolor{blue!15}$4.49$
		& \cellcolor{blue!15}$6.22$ & $8.65$
        & $12.1$ \\ \hline
    \end{tabular}
\end{center}
and we see that in terms of expected number of rounds \RPSLS\
is more efficient than \CTLS\ when $n\le 6$.

Extensions of \RPSLS\ to more hands than three were also
studied in distributed computing contexts by Suzaki and Osaki
in several papers (see \cite{SuOs2,SuOs3,SuOs}), where partial
probabilistic analyses are provided; see also Section~\ref{examples}.

Our aim in this paper is to propose a general model containing all
the schemes above as special cases. More precisely, we will define a
wide class of generalized Janken games with $m\ge 2$ hands, and we
will show that the different behaviors mentioned above for \CTLS\ and
\RPSLS\ are prototypical and find their extensive form in a general
setting. More precisely, we will show that the complexity of leader
selection based on general Janken schemes exhibits a gap theorem: the
average number of rounds needed to select a leader is either of
\emph{logarithmic} order or of \emph{exponential} order. Moreover, we
will establish stronger results including bounded variance and the
oscillations of the asymptotic distribution in the log-case, and an
exponential limit law in the exp-case.

This paper is structured as follows. In the next section, we
define our generalized Janken games and provide general tools for
analyzing the number of rounds until a leader is selected. In
Section~\ref{examples}, we apply our results to some games.
Finally, in Section~\ref{ext}, we discuss extensions and other
gap theorems in our model. We conclude the paper with a remark
on infinite-hand games.

\section{Dichotomous behavior of the complexity}

We define first the class of generalized Janken games we will
analyze. Then we distinguish between two subclasses of games
(log-games and exp-games) for which the cost complexity differs
significantly from being logarithmic to being exponential. We will
also derive more precise asymptotic approximations.

The analysis of log-games will be more subtle due to the inherent
periodic fluctuations in the asymptotic approximations to the moments
and the limit law, which will be clarified by generalizing our recent
approach from \cite{FuHwMaIt}. The analysis of exp-games, on the
other hand, is more straightforward and our results will follow by the
method of moments.

\subsection{Generalized Janken games}

Let $m\ge 2$, and we are given $m$ hands $\{\mathscr{H}_1,\ldots,
\mathscr{H}_m\}$ with (positive) probabilities $\{p_1,\ldots,p_m\}$.
A generalized Janken game is played as follows. Assume that the set
of hands the $n$ players may choose equals $S=\{\mathscr{H}_{i_1},
\ldots, \mathscr{H}_{i_\ell}\}$. Then there are two situations:

\begin{itemize}

\item $S$ is a (clear-cut) \emph{win-or-defeat (abbreviated as WOD)
set}, i.e., $S=W\cup D$ with $W\cap D=\emptyset$, where
$W=\{\mathscr{H}_{j_1}, \ldots, \mathscr{H}_{j_r}\}, 1\leq r<\ell$ is a set
of winning hands, meaning that players having chosen these hands
continue to play in the next round, and $D=\{\mathscr{H}_{k_1},
\ldots,\mathscr{H}_{k_{\ell-r}}\}$ is a set of losing hands, meaning
that players having chosen these hands are eliminated after the
current round.

\item The hands in $S$ result in a \emph{tie}, meaning that no one is
eliminated and the round is non-conclusive and has to be repeated.

\end{itemize} The generalized Janken game is played one round after
another until a single player remains who is then the leader. This is
always possible if there is at least one WOD set whose cardinality is
two. We refer to these games as \GJLS.

For analysis purposes, we introduce more notations. First, given
$S=\{\mathscr{H}_{i_1},\ldots,\mathscr{H}_{i_\ell}\}$, define
\[
    \pi_n^{(S)}:=\sum_{\substack{j_1+\cdots+j_\ell=n\\
	j_1,\ldots,j_\ell\ge 1}}
	\binom{n}{j_1,\ldots,j_\ell}p_{i_1}^{j_1}
	\cdots p_{i_\ell}^{j_\ell}.
\]
Then, by the inclusion-exclusion principle,
\begin{align*}
    \pi_n^{(S)}
    &=\left(p_{i_1}+\cdots+p_{i_{\ell}}\right)^n\\
	&\quad -\left(p_{i_2}+\cdots+p_{i_\ell}\right)^n-\cdots-
    \left(p_{i_1}+\cdots+p_{i_{\ell-1}}\right)^n\\
    &\quad +\left(p_{i_3}+\cdots+p_{i_\ell}\right)^n
    +\left(p_{i_2}+p_{i_4}+\cdots+p_{i_\ell}\right)^n+\cdots+
    \left(p_{i_1}+\cdots+p_{i_{\ell-2}}\right)^n\\
	&\quad\pm\cdots.
\end{align*}
From this, we see that, for large $n$,
\begin{equation}\label{asym-prob}
    \pi_n^{(S)}
	\sim \left(p_{i_1}+\cdots+p_{i_{\ell}}\right)^n.
\end{equation}
Let $\WD$ denote the set of all WOD sets. Define now two game indices:
\begin{equation}\label{def-rho}
    \rho
	:=\max_{\{\mathscr{H}_{i_1},\ldots,
	\mathscr{H}_{i_\ell}\}\in\WD}
	\{p_{i_1}+\cdots+p_{i_\ell}\},
\end{equation}
and $\nu$ the number of WOD sets attaining the maximum value $\rho$.
We distinguish between two cases:
\begin{itemize}
\item log-game: $\rho=1$;
\item exp-game: $\rho<1$.
\end{itemize}
Note that a log-game occurs if and only if $\{\mathscr{H}_1,\ldots,
\mathscr{H}_m\}$ is itself a WOD set, and in this case
$\nu=1$. Also log-games are more meaningful when the hands are
generated by purely random mechanisms but not by intentional
calculation.

We are interested in the number of rounds $X_n$ used by $n$ players
to select a leader by \GJLS, which is one of the most important cost
measures of the game. This number satisfies, by considering the
size of the winning group after the first round of \GJLS, the
following distributional recurrence
\begin{equation}\label{dis-rec-1}
    X_n\stackrel{d}{=}X_{I_n}+1\qquad (n\ge 2),
\end{equation}
where the $X_n$'s and $I_n$'s are independent, $X_1=0$, and
for $1\le j\le n$,
\begin{align}\label{In}
	\mathbb{P}(I_n=j)
	=\begin{cases} 1-\varpi_n,&\text{if}\ j=n;\\
	\displaystyle
	\sum_{\substack{S\in\WD\\ S=W\cup D}}
	\binom{n}{j}\pi_j^{(W)}\pi_{n-j}^{(D)},
	&\text{if}\ 1\le j<n,
	\end{cases}
\end{align}
where
\[
	\varpi_n
	:=\sum_{S\in\WD}\pi_n^{(S)}
\]
stands for the probability of no tie occurring. From
(\ref{asym-prob}) and by the definitions of $\rho$ and $\nu$, we have
\begin{align}\label{prob-no-tie}
    \varpi_n\sim \nu\rho^n.
\end{align}

Alternatively, instead of considering only one round, one may also
wait for a random number of times $T_n$ until a WOD set is reached.
This then yields the following alternative distributional recurrence
for $X_n$
\begin{equation}\label{dis-rec-2}
    X_n\stackrel{d}{=}X_{J_n}+T_n\qquad (n\ge 2),
\end{equation}
where the $X_n$'s, $J_n$'s, and $T_n$'s are independent, $T_n$ is a
geometric distributed random variable with parameter $\varpi_n$,
$X_1=0$ and for $1\le j<n$,
\begin{align}\label{Jn}
	\mathbb{P}(J_n=j)=\frac{1}{\varpi_n}
	\sum_{\substack{S\in\WD\\ S=W\cup D}}
	\binom{n}{j}\pi_j^{(W)}\pi_{n-j}^{(D)}.
\end{align}

Both forms of the distributional recurrence will be useful for us;
(\ref{dis-rec-1}) will be used in the analysis of log-games, whereas
(\ref{dis-rec-2}) is advantageous in the analysis of exp-games.

\subsection{Log-games}\label{log-case}

In this subsection, we consider log-games for which $\rho=1$ and
$\nu=1$. In this case the whole set of hands $\{\mathscr{H}_1,
\ldots, \mathscr{H}_m\}$ is itself a WOD set, and we define
\[
    \alpha:=\sum_{\mathscr{H}_j\text{ is a winning hand in }
    \{\mathscr{H}_1,\ldots,\mathscr{H}_m\}}p_j.
\]

Denote by $X_n$ the number of rounds to select a leader by \GJLS.
Such games are marked by its complexity $X_n$ satisfying logarithmic
mean, bounded variance and periodic oscillations of the asymptotic
distributions. This is the same pattern as for \CTLS; see Section \ref{intro}.
For the proofs, we extend the analytic approach used in our previous paper \cite{FuHwMaIt}.

\paragraph{Mean value.}
Let $\mu_n := \mathbb{E}(X_n)$. Then, by (\ref{dis-rec-1}),
we obtain that
\[
	\mu_n
	=\sum_{1\le j<n}\sum_{\substack{S\in\WD\\ S=W\cup D}}
	\binom{n}{j}\pi_j^{(W)}\pi_{n-j}^{(D)}
	\mu_j+(1-\varpi_n)\mu_n+1,
\]
for $n\ge 2$, with the initial condition $\mu_1=0$.

To solve this
recurrence, we consider first the Poisson generating function
\[
    \tilde{f}_1(z)
    :=e^{-z}\sum_{n\ge 1}\mu_n\frac{z^n}{n!}.
\]

In what follows, an ``exponentially small term'' is used to mean an
entire function that is bounded above by $e^{-c\Re(z)}$ for some
$c>0$ as $|z|\to\infty$ in the half-plane $\Re(z)>0$. \begin{lmm} The
Poisson generating function of $\mu_n$ satisfies
the functional equation
\begin{align}\label{f1z-PGF}
    \tilde{f}_1(z)
    =\tilde{f}_1(\alpha z)+1+\sum_{1\le j\le k}
    \lambda_je^{-\beta_jz}\tilde{f}_1(\alpha_jz)-(1+z)e^{-z},
\end{align}
with $\tilde{f}_1(0)=0$, where $\lambda_j\in\{-1,1\}$ and
$0<\alpha_j,\beta_j<1$ are constants.
\end{lmm}
\begin{proof}
Since $\rho=1$, we see that $1-\varpi_n$ consists only of
exponentially small terms. This implies that we can arrange the terms
and write
\[
    e^{-z}\sum_{n\ge 1}(1-\varpi_n)
    \mathbb{E}(X_n)\frac{z^n}{n!}
    =\sum_{1\le j\le k}\lambda_j
    e^{-\beta_jz}\tilde{f}_1(\alpha_jz).
\]
On the other hand
\begin{align*}
    e^{-z}\sum_{n\ge 2}\sum_{1\le j<n}
    &\sum_{\substack{S\in\WD\\ S=W\cup D}}
	\binom{n}{j}\pi_j^{(W)}\pi_{n-j}^{(D)}
	\mathbb{E}(X_j)\frac{z^n}{n!}\\
    &=e^{-z}\sum_{\substack{S\in\WD\\ S=W\cup D}}
	\left(\sum_{n\ge 1}\pi_{n}^{(D)}\frac{z^n}{n!}\right)
	\left(\sum_{n\ge 1}\pi_n^{(W)}
	\mathbb{E}(X_n)\frac{z^n}{n!}\right).
\end{align*}
The largest terms (as $|z|\to\infty$ in $\Re(z)>0$) comes from the
whole set $S=\{\mathscr{H}_1,\ldots,\mathscr{H}_m\}$, which is itself
a WOD set and produces terms of the form (by
\eqref{asym-prob})
\[
	e^{-z}\left(\sum_{n\ge 1}
	\pi_{n}^{(D)}\frac{z^n}{n!}\right)
	\left(\sum_{n\ge 1}\pi_n^{(W)}
	\mathbb{E}(X_n)\frac{z^n}{n!}\right)
	=\tilde{f}_1(\alpha z)
	+\sum_{1\le j\le k}
	\lambda_je^{-\beta_jz}\tilde{f}_1(\alpha_jz),
\]
for some $\lambda_j\in\{-1,1\}$ and $0<\alpha_j,\beta_j<1$ (whose
values may differ from one occurrence to another). For all other
WOD sets different from $\{\mathscr{H}_1,\ldots,\mathscr{H}_m\}$,
we have $\sum_{\mathscr{H}_i\in S}p_i<1$ and they can be regrouped as
\[
	e^{-z}\sum_{\substack{S\in\WD\\ S=W\cup D\\
	S\ne\{\mathscr{H}_1,\ldots,\mathscr{H}_m\}}}
	\left(\sum_{n\ge 1}\pi_{n}^{(D)}\frac{z^n}{n!}\right)
	\left(\sum_{n\ge 1}\pi_n^{(W)}
	\mathbb{E}(X_n)\frac{z^n}{n!}\right)
	=\sum_{1\le j\le k}\lambda_j
	e^{-\beta_jz}\tilde{f}_1(\alpha_jz),
\]
where $\lambda_j\in\{-1,1\}$ and $0<\alpha_j,\beta_j<1$. The
remaining computations are straightforward. Thus, our claim is
proved.
\end{proof}

We will write \eqref{f1z-PGF} as
\begin{equation}\label{func-eq-mean}
    \tilde{f}_1(z)=\tilde{f}_1(\alpha z)+1+\tilde{\phi}_1(z),
\end{equation}
where $\tilde{\phi}_1(z)$ is exponentially small in the half-plane
$\Re(z)>0$.

\paragraph{Variance.} Consider now the Poisson generating function of
$\mathbb{E}(X_n^2)$, denoted by
\[
    \tilde{f}_2(z)
    :=e^{-z}\sum_{n\ge 1}\mathbb{E}(X_n^2)\frac{z^n}{n!}.
\]
Define
\[
    \tilde{V}(z):=\tilde{f}_2(z)-\tilde{f}_1(z)^2.
\]
\begin{lmm} The function $\tilde{V}(z)$ satisfies the functional
equation
\begin{equation}\label{func-eq-var}
    \tilde{V}(z)=\tilde{V}(\alpha z)+\tilde{\phi}_3(z),
\end{equation}
where $\tilde{\phi}_3(z)$ is an exponentially small term.
\end{lmm}
\begin{proof}
A similar analysis as that used for $\tilde{f}_1(z)$ leads to the
functional equation
\begin{equation}\label{func-eq-sm}
    \tilde{f}_2(z)=\tilde{f}_2(\alpha z)
	+2\tilde{f}_1(\alpha z)+1+\tilde{\phi}_2(z),
\end{equation}
where $\tilde{\phi}_2(z)$ consists of exponentially small terms
(involving both $\tilde{f}_1(z)$ and $\tilde{f}_2(z)$). Then
\eqref{func-eq-var} follows from (\ref{func-eq-mean}) and
(\ref{func-eq-sm}).
\end{proof}

\paragraph{Asymptotics and JS-admissibility.} From the Poisson
generating functions $\tilde{f}_1(z)$ (which is indeed the expected
cost $X_n$ when $n$ itself follows a Poisson($z$) distribution), we
can recover the asymptotic behaviors of $\mathbb{E}(X_n)$ by the
relation
\[
    \mu_n = n![z^n]e^z \tilde{f}_1(z).
\]
This can be done by several means, and a by-now standard approach are
the so-called analytic de-Poissonization techniques largely developed
by Jacquet and Szpankowski in \cite{JacSz}, which rely on the
saddle-point method (see \cite{FS09}). The use of such an approach
can be further schematized by introducing the notion of
$\JS$-admissible functions, which we formulated in \cite{HwFuZa}. We
briefly sketch the underlying ideas, and refer the interested readers
to our previous papers for more details \cite{FuHwMaIt,HwFuZa}.
Similarly, since $\mu_n$ is of logarithmic order and the variance is
bounded, the function $\tilde{V}(z)$ will provide a sufficiently good
approximation to the variance $\mathbb{V}(X_n)$.

We recall the following definition from \cite{HwFuZa}.
Let
\[
    \mathscr{C}_{\varepsilon} :=
	\{z\,:\, |\arg(z)| \le \varepsilon\},
\]
where $\varepsilon>0$ is an arbitrary constant (which will
subsequently be used as a generic symbol whose value may change
from one occurrence to another).
\begin{Def}
Let $\tilde{f}(z)$ be an entire function and $\xi,
\eta\in\mathbb{R}$. Then $\tilde{f}(z)$ is \emph{JS-admissible},
written $\tilde{f}\in\JS$ (or more precisely, $\tilde{f} \in
\JS_{\!\!\xi,\eta})$, if for $0<\phi<\pi/2$ and all $\vert z\vert\ge
1$ the following two conditions hold.
\begin{itemize}
\item[\textbf{(I)}] Uniformly for $z\in\mathscr{C}_{\varepsilon}$,
\[
	\tilde{f}(z)
	=O\left(\vert z\vert^{\xi}
	(\log_{+}\vert z\vert)^{\eta}\right),
\]
where $\log_{+}x:=\log(1+x)$.
\item[\textbf{(O)}] Uniformly for $\phi\le\vert\arg(z)\vert\le\pi$,
\[
	f(z):=e^{z}\tilde{f}(z)
	=O\left(e^{(1-\varepsilon)\vert z\vert}\right),
\]
for some $\varepsilon>0$.
\end{itemize}
\end{Def}
Such functions enjoy closure properties under several different
operations, and it is these properties that make such a notion really
useful. Also if $f\in\JS_{\!\!\xi,\eta}$, then
\[
    n![z^n]f(z) = \sum_{0\le j<2k}
	\frac{\tilde{f}^{(j)}(n)}{j!}\,\tau_j(n)
	+O\left(n^{\xi-k} (\log n)^\eta
	\right)\qquad(k\ge1),
\]
where
\[
    \tau_j(n) := \sum_{0\le \ell \le j}
	\binom{j}{\ell}(-1)^{j-\ell}
	\frac{n!n^{j-\ell}}{(n-\ell)!}
	\qquad(j=0,1,\dots)
\]
are Charlier polynomials. In particular,
\begin{align}\label{poi-heu}
    n![z^n]f(z) = \tilde{f}(n) - \frac{\tilde{f}''(n)}2\,n
	+ \frac{\tilde{f}'''(n)}3\, n +O\left(
	n^{\xi-2}(\log n)^\eta\right).
\end{align}
Specific to our analysis, we need additionally the following
property.

\begin{lmm}\label{transfer-JS}
Let $\tilde{f}$ and $\tilde{g}$ be two entire functions satisfying
the functional equation
\begin{equation}\label{gen-func-eq}
	\tilde{f}(z)=\tilde{f}(\alpha z)
	+\sum_{1\le j\le k}\lambda_je^{-\beta_jz}
	\tilde{f}(\alpha_jz)+\tilde{g}(z),
\end{equation}
with $\tilde{f}(0)=\tilde{g}(0)=0$, where $\lambda_j\in\mathbb{R}$,
$\beta_j>0$ and $\alpha,\alpha_j\in(0,1)$. Then
\[
	\tilde{g}\in\JS
	\quad\Longleftrightarrow
	\quad\tilde{f}\in\JS.
\]
\end{lmm}
The proof is similar to that of Proposition 3 in the Appendix of
\cite{FuHwMaIt}, and is omitted here.

From this and (\ref{func-eq-mean}), (\ref{func-eq-sm}), we see that
$\tilde{f}_1\in\JS_{\!\!0,1}$ and $\tilde{f}_2\in\JS_{\!\!0,2}$. By
\eqref{poi-heu} and the arguments used for the variance in
\cite{HwFuZa}, we obtain that
\begin{equation}\label{dep-mean-var}
	\mathbb{E}(X_n)=\tilde{f}_1(n)+O\left(n^{-1}\right),\qquad
	\mathbb{V}(X_n)=\tilde{V}(n)+O\left(n^{-1}\right).
\end{equation}
Thus for asymptotic purposes, we can entirely focus on the Poisson
model. A standard approach for the asymptotics of the Poisson
generating function for large $|z|$ in similar situations is based
on the Mellin transform techniques
\[
	F^{*}(s)=\mathscr{M}[\tilde{f};s]
	:=\int_{0}^{\infty}\tilde{f}(z)z^{s-1}{\rm d}z;
\]
see \cite{FlGoDu} for an authoritative survey. We first derive the
asymptotic behavior for functions satisfying the more general
equation (\ref{gen-func-eq}). Let
\[
	\tilde{\phi}(z)
	:=\sum_{1\le j\le k}\lambda_j
	e^{-\beta_jz}\tilde{f}(\alpha_jz).
\]
\begin{pro}\label{asym-gen-func-eq}
Let $L:=\log(1/\alpha)$ and $\chi_k:=2k\pi i/L$.
\begin{itemize}
	
\item[(i)] Assume that $\tilde{g}\in\JS_{\xi,\eta}$ with
$\xi<0$. Then, as $|z|\to\infty$ in the sector
$\mathscr{C}_{\varepsilon}$,
\begin{align}\label{fz-i}
	\tilde{f}(z)= \frac{1}{L}\sum_{k\in\mathbb{Z}}
	\left(\Phi^{*}(\chi_k)+G^{*}(\chi_k)\right)z^{-\chi_k}
	+O\left(|z|^{-\min\{1,\xi\}}
	(\log|z|)^\eta\right),
\end{align}
where $\Phi^{*}(s):=\mathscr{M}[\tilde{\phi};s]$ and $G^{*}(s)
:=\mathscr{M}[\tilde{g};s]$.

\item[(ii)] Assume that $\tilde{g}\in\JS$ and
\[
	\tilde{g}(z)
	=c+O\left(\vert z\vert^{-\xi}\right),
\]
as $\vert z\vert\to\infty$ in the sector $\mathscr{C}_{\varepsilon}$.
Then
\begin{align*}
	\tilde{f}(z)
	&= c\log_{1/\alpha} z+\frac{c}{2}+\frac{d+\Phi^{*}(0)}{L}
	+\sum_{k\ne 0}\left(\Phi^{*}(\chi_k)
	+G^{*}(\chi_k)\right)z^{-\chi_k}
	+O\left(|z|^{-\min\{1,\xi\}}\right),
\end{align*}
as $|z|\to\infty$ in $\mathscr{C}_{\varepsilon}$, where
$d=\lim_{s\to 0}(G^{*}(s)+c/s)$.
\end{itemize}
\end{pro}
\begin{Rem}\label{exp-four}
Note that Fourier coefficients of the periodic functions in the above
result are less explicit due to the occurrence of $\tilde{f}(z)$ in
$\tilde{\phi}(z)$. However, in some situations (e.g., \CTLS) they can
be made explicit; see \cite{Pr} and the remarks in
Section~\ref{examples}.
\end{Rem}

\vspace*{0.2cm}\pf Consider part \emph{(i)}. The assumptions imply
that $G^{*}(s)$ exists in the strip $-1<\Re(s)<-\xi$ and
\begin{equation}\label{meltr-F}
	F^{*}(s):=\mathscr{M}[\tilde{f};z]
	=\frac{\Phi^{*}(s)+G^{*}(s)}{1-\alpha^{-s}},
	\qquad \Re(s)\in(-1,0).
\end{equation}
Note that from the assumptions, Lemma~\ref{transfer-JS} and the
Exponential Smallness Lemma in \cite{FuHwZa}, we obtain that
$F^{*}(s)$ decays exponentially fast along vertical lines in
$-1<\Re(s)<-\xi$. Consequently, by standard Mellin argument,
we deduce \eqref{fz-i}.

For part \emph{(ii)}, by the Direct Mapping Theorem in \cite{FlGoDu},
we see that $G^{*}(s)$ has a simple pole at $s=0$ with the
singularity expansion $G^{*}(s)\asymp c/s+d+\cdots$. The rest of the
proof then follows as in part \emph{(i)}. \qed

Collecting all results, we then derive the asymptotics of the mean
and the variance.

\begin{thm}[Log-games]\label{main-thm1-log}
If $\rho=1$, then the number of rounds $X_n$ to find a leader used by
\GJLS\ satisfies
\[
	\mathbb{E}(X_n)=\log_{1/\alpha}n
	+P_1(\log_{1/\alpha}n)+ O\left(n^{-1}\right),
\]
and
\[
	\mathbb{V}(X_n)
	= P_2(\log_{1/\alpha}n)+ O\left(n^{-1}\right),
\]
where $P_1(t),P_2(t)$ are bounded, $1$-periodic functions.
\end{thm}

\paragraph*{Limit law.} We now turn to the limit law. We begin with
considering the Poisson generating function of the probability
generating function of $X_n$
\[
	\tilde{P}(y,z)
	:=e^{-z}\sum_{n\ge 1}\mathbb{E}
	\left(y^{X_n}\right) \frac{z^n}{n!},
\]
which satisfies the functional equation
\[
	\tilde{P}(y,z)=y\tilde{P}(y,\alpha z)
	+y\sum_{1\le j\le k}\lambda_je^{-\beta_jz}
	\tilde{P}(y,\alpha_jz)+(1-y)ze^{-z},
\]
where $\lambda_j\in\{-1,1\}$ and $\alpha,\alpha_j,\beta_j
\in(0,1)$. Observe that
\[
\frac{\tilde{P}(y,z)}{1-y}=\sum_{\ell\ge 0}\tilde{A}_{\ell}(z)y^{\ell},
\]
where
\[
	\tilde{A}_{\ell}(z)
	:=e^{-z}\sum_{n\ge 0}\mathbb{P}(X_n\le\ell)\frac{z^n}{n!}.
\]
By dividing the above functional equation by $1-y$ and reading off
coefficients, we obtain the recursive functional equation
\begin{equation}\label{func-eq-Al}
	\tilde{A}_{\ell+1}(z)
	=\tilde{A}_{\ell}(\alpha z)
	+\sum_{1\le j\le k}\lambda_je^{-\beta_jz}
\tilde{A}_{\ell}(\alpha_jz),\qquad (\ell\ge 0)
\end{equation}
with $\tilde{A}_0(z)=ze^{-z}$.

The equations here are similar to
those in \cite{FuHwMaIt}, and the remaining analysis follows along
the same lines there, which we now sketch. First, define
$\tilde{R}_0(z):=\tilde{A}_0(z)$ and for $\ell\ge 0$
\[
	\tilde{R}_{\ell+1}(z)
	:=\sum_{1\le j\le k}\lambda_j
	e^{-\beta_j z}\tilde{A}_{\ell}(\alpha_jz).
\]
Then
\[
	\tilde{A}_{\ell+1}(z)
	=\tilde{A}_{\ell}(\alpha z)+\tilde{R}_{\ell+1}(z),
\]
which, by iteration, leads to
\[
	\tilde{A}_{\ell}(z)
	=\sum_{0\le j\le \ell}\tilde{R}_{j}
	\bigl(\alpha^{\ell-j}z\bigr).
\]

For the asymptotics, we need the JS-admissibility in a stronger
uniform form.
\begin{lmm} For $\ell\ge1$,
$\tilde{A}_{\ell}(z)$ is uniformly JS-admissible, i.e., for $\vert\arg(z)\vert\le\varepsilon, 0<\varepsilon<\pi/2$,
\begin{equation}\label{I}
	\tilde{A}_{\ell}(z)
	=O\left(\vert z\vert^{\varepsilon_1}\right),
\end{equation}
uniformly in $z$, and, for $\varepsilon\le\vert\arg(z)\vert\le\pi$,
\begin{equation}\label{O}
	e^{z}\tilde{A}_{\ell}(z)
	=O\bigl(e^{(1-\varepsilon_2)\vert z\vert}\bigr),
\end{equation}
uniformly in $z$. Here the involved constants in both cases are
absolute and $\varepsilon_1, \varepsilon_2$ are positive constants.
\end{lmm}
\pf The proof is similar to that of Lemma 3 in \cite{FuHwMaIt}.
\qed

Then, by standard de-Poissonization arguments \cite{HwFuZa,JacSz},
we deduce that
\[
	\mathbb{P}(X_n\le\ell)
	=\sum_{0\le j\le \ell}\tilde{R}_j
	\bigl(\alpha^{\ell-j}n\bigr)
	+O\left(n^{-1+\varepsilon} \right),
\]
uniformly in $\ell$. On the other hand, by the definition of
$\tilde{R}_j(x)$, we see that $\vert \tilde{R}_j(x)\vert\le e^{-px}$
for some $p>0$. Consequently,
\[
	\sum_{j>\ell}\tilde{R}_j\bigl(\alpha^{\ell-j}n\bigr)
	\le\sum_{j>\ell}e^{-p\alpha^{\ell-j}n}
	=\sum_{j\ge 1}e^{-p\alpha^{-j}n}=O\left(e^{-pn}\right).
\]
The latter implies that
\[
	\mathbb{P}(X_n\le\ell)=\sum_{j\ge0}
	\tilde{R}_j\bigl(\alpha^{\ell-j}n\bigr)
	+O\left(n^{-1+\varepsilon} \right),
\]
uniformly in $\ell$.

Replacing now $\ell$ by $\lfloor\log_{1/\alpha}n\rfloor+\ell$,
and writing $\vartheta(n)=\{\log_{1/\alpha}n\}$, we obtain
\[
	\mathbb{P}(X_n\le\lfloor\log_{1/\alpha}n\rfloor+\ell)
	=\sum_{j\ge0}\tilde{R}_j
	\bigl(\alpha^{-\vartheta(n)+\ell-j}\bigr)
	+O\left(n^{-1+\varepsilon} \right).
\]
We summarize the analysis as follows.

\begin{thm}[Log-games] \label{main-thm2-log}
If $\rho=1$, then the number $X_n$ of rounds used by \GJLS\ until a
leader is selected satisfies
\[
	\mathbb{P}(X_n\le\lfloor\log_{1/\alpha}n\rfloor+\ell)
	=\sum_{j\ge0}\tilde{R}_j
	\bigl(\alpha^{-\vartheta(n)+\ell-j}\bigr)
	+O\left(n^{-1}\right),
\]
uniformly in $\ell$.
\end{thm}
The improved error term comes from refining the above analysis
(by expanding more terms in the de-Poissonization procedure).

\begin{Rem}\label{exp-ll}
It is possible to derive more explicit expressions for
$\tilde{A}_{\ell}(z)$ and $\tilde{R}_j(z)$ in
Theorem~\ref{main-thm2-log}, but they are generally messy.
For that purpose, rewrite (\ref{func-eq-Al}) as
\[
	\tilde{A}_{\ell+1}(z)=\sum_{0\le j\le k}\lambda_j
	e^{-\beta_jz}\tilde{A}_{\ell}(\alpha_jz),
\]
where $\lambda_0=1$, $\alpha_0=\alpha$ and $\beta_0=0$. Iterating
yields
\[
	\tilde{A}_{\ell}(z)
	=\sum_{\mathbf{j}=(j_1,\ldots,j_{\ell})
	\in\{0,\ldots,k\}^{\ell}}
	\lambda_{\mathbf j}
	e^{-\sum_{i=1}^{\ell}\alpha_{j_1}
	\cdots\alpha_{j_{i-1}}\beta_{j_i}z}
	\tilde{A}_0({\alpha}_{\mathbf j} z),
\]
where $\lambda_{\mathbf{j}}=\lambda_{j_1} \lambda_{j_2}
\cdots\lambda_{j_{\ell}}$, and ${\alpha}_{\mathbf{j}}
=\alpha_{j_1}\alpha_{j_2}\cdots\alpha_{j_{\ell}}$.
Since $\tilde{A}_0(z)=ze^{-z}$, we then have
\[
	\tilde{A}_{\ell}(z)
	=z\sum_{\mathbf{j}=(j_1,\ldots,j_{\ell})
	\in\{0,\ldots,k\}^{\ell}}{\alpha}_{\mathbf j}
	\lambda_{\mathbf j}e^{-\sum_{i=1}^{\ell+1}\alpha_{j_1}
	\cdots\alpha_{j_{i-1}}\beta_{j_i}z},
\]
with $j_{\ell+1}:=1$. We will give some examples for which this
expression simplifies in Section~\ref{examples}.
\end{Rem}

\subsection{Exp-games}\label{exp-case}

In this section we consider the exp-games for which $\rho<1$ and the
probability $\varpi_n$ of tie occurring satisfies $\varpi_n\sim
\nu\rho^n$. We show that $X_n$ converges (with all its moments) to an
exponential distribution.

\begin{thm}[Exp-games]\label{main-thm-exp}
Assume $\rho<1$. The number of rounds $X_n$ used to select a leader
by \GJLS\ converges in distribution and with all moments to an
exponential distribution:
\[
   \nu\rho^nX_n\stackrel{d}{\longrightarrow}{\rm Exp}(1),
\]
where ${\rm Exp}(1)$ denotes an exponential random variable with
mean one. In particular,
\[
    \mathbb{E}(X_n)\sim\nu^{-1}\rho^{-n} .
\]
\end{thm}
The proof relies on the recurrence \eqref{dis-rec-2} using the method
of moments.

First, taking the $m$-th moment on both sides of (\ref{dis-rec-2})
gives ($\mu_{n,m} := \mathbb{E}\left(X_n^{m}\right)$)
\begin{align}
	\mu_{n,m}
	&=\sum_{0\le k\le m}\binom{m}{k}
	\mathbb{E}\left(T_n^{m-k}\right)
	\mu_{J_n,k}\nonumber\\
	&=\frac{1}{\varpi_n}\sum_{0\le k\le m}
	\binom{m}{k}\mathbb{E}\left(T_n^{m-k}\right)
	\sum_{1\le j<n}\sum_{\substack{S\in\WD\\ S=W\cup D}}
	\binom{n}{j}\pi_j^{(W)}\pi_{n-j}^{(D)}
	\mu_{j,k}\label{rec-mth-mom}.
\end{align}
We single out the terms corresponding to $k=m$ and have
\begin{align}
	\mu_{n,m}
	&=\frac{1}{\varpi_n}\sum_{1\le j<n}
	\sum_{\substack{S\in\WD\\ S=W\cup D}}
	\binom{n}{j}\pi_j^{(W)}\pi_{n-j}^{(D)}
	\mu_{j,m}\nonumber\\
	&\qquad+\frac{1}{\varpi_n}\sum_{0\le k<m}
	\binom{m}{k}\mathbb{E}\left(T_n^{m-k}\right)
	\sum_{1\le j<n}\sum_{\substack{S\in\WD\\ S=W\cup D}}
	\binom{n}{j}\pi_j^{(W)}\pi_{n-j}^{(D)}
	\mu_{j,k}.\label{rec-mth-mom-1}
\end{align}
The proof proceeds in two steps: first, we derive an upper bound for
$\mu_{n,m}$, and then we refine it and get a more precise asymptotic
approximation to $\mu_{n,m}$.

\begin{lmm}\label{bound-Tn}
For $m\ge 0$,
\[
    \mathbb{E}(T_n^{m})
	=O\left(\rho^{-nm}\right).
\]
\end{lmm}
\pf This follows from the fact that $\mathbb{E}(T_n^m)
=P\left(\varpi_n^{-1}\right)$, where $P(x)$ is a polynomial of degree
$m$ (without constant term), and the asymptotics \eqref{prob-no-tie}
of $\varpi_n$. \qed

We now use this to find a similar bound for $\mu_{n,m}$.
\begin{lmm}
For $m\ge 0$,
\[
    \mu_{n,m}=O\left(\rho^{-nm}\right).
\]
\end{lmm}
\pf We prove by induction that
\begin{align} \label{munm-bd1}
    \mu_{n,m}\le c_m\rho^{-nm},
\end{align}
for some constants $c_m$.

Assume that the bound \eqref{munm-bd1} is proved for all moments of
order $<m$ and for the $m$-th moment for all indices $<n$. To prove
it for $n$, we use (\ref{rec-mth-mom-1}). First, observe that
\begin{align*}
	\frac{1}{\varpi_n}&\sum_{0\le k<m}
	\binom{m}{k}\mathbb{E}(T_n^{m-k})
	\sum_{1\le j<n}\sum_{\substack{S\in\WD\\ S=W\cup D}}
	\binom{n}{j}\pi_j^{(W)}\pi_{n-j}^{(D)}\mu_{j,k}\\
	&=O\left(\frac{1}{\varpi_n}
	\sum_{0\le k<m}\rho^{-n(m-k)}\sum_{1\le j<n}
	\sum_{\substack{S\in\WD\\ S=W\cup D}}
	\binom{n}{j}\pi_j^{(W)}\pi_{n-j}^{(D)}
	\rho^{-jk}\right)\\
	&=O\left(\rho^{-nm}\right).
\end{align*}
Thus
\[
	\frac{1}{\varpi_n}\sum_{0\le k<m}
	\binom{m}{k}\mathbb{E}(T_n^{m-k})
	\sum_{1\le j<n}\sum_{\substack{S\in\WD\\ S=W\cup D}}
	\binom{n}{j}\pi_j^{(W)}\pi_{n-j}^{(D)}\mu_{j,k}
	\le d_m\rho^{-nm}
\]
for a suitable constant $d_m$. By this bound, induction hypothesis
and \eqref{rec-mth-mom-1}, we obtain that
\begin{align*}
	\mu_{n,m}
	&\le\frac{c_m}{\varpi_n}\sum_{1\le j<n}
	\sum_{\substack{S\in\WD\\ S=W\cup D}}
	\binom{n}{j}\pi_j^{(W)}\pi_{n-j}^{(D)}
	\rho^{-jm}+d_m\rho^{-nm}\\
	&\le\left(\rho^mc_m+d_m\right)\rho^{-nm}
	\le c_m\rho^{-nm},
\end{align*}
where the last step follows by choosing $c_m$ such that $c_m\ge
d_m/(1-\rho^m)$.\qed

Now we refine Lemma \ref{bound-Tn}.
\begin{lmm}\label{ref-toll}
For $m\ge 0$,
\begin{align} \label{Tnm-bd}
	\mathbb{E}(T_n^{m})
	\sim \frac{m!}{\nu^m\rho^{nm}}.
\end{align}
\end{lmm}
\pf The moment generating function of $\varpi_nT_n$ is given by
\[
	\mathbb{E}(e^{\varpi_nT_nt})
	=\frac{\varpi_ne^{\varpi_nt}}
	{1-(1-\varpi_n)e^{\varpi_nt}}.
\]
Since $\varpi_n\to 0$, we see that
\[
    \mathbb{E}(e^{\varpi_nT_nt})\to\frac{1}{1-t},
\]
for $|t|<1$. The latter is the moment generating function of an
exponential distribution with mean $1$. Consequently, we have
\[
    \mathbb{E}(\varpi_nT_n)^m\to m! \qquad(m\ge0),
\]
which, together with the asymptotics of $\varpi_n$, proves
\eqref{Tnm-bd}.\qed

We now prove a similar result for $\mu_{n,m}$. Rewrite
\eqref{rec-mth-mom} as follows.
\begin{equation}\label{rec-mth-mom-2}
	\mathbb{E}(X_n^{m})=\mathbb{E}(T_n^{m})
	+\frac{1}{\varpi_n}\sum_{1\le k\le m}
	\binom{m}{k}\mathbb{E}(T_n^{m-k})\sum_{1\le j<n}
	\sum_{\substack{S\in\WD\\ S=W\cup D}}
	\binom{n}{j}\pi_j^{(W)}\pi_{n-j}^{(D)}
	\mu_{j,k}.
\end{equation}
\begin{lmm}
For $m\ge 0$,
\begin{align}\label{munm-bd2}
	\mu_{n,m}
	\sim \frac{m!}{\nu^m\rho^{nm}}.
\end{align}
\end{lmm}
\pf We use (\ref{rec-mth-mom-2}). First the second term on the right
hand side of (\ref{rec-mth-mom-2}) is bounded above by
\begin{align*}
	\frac{1}{\varpi_n}&\sum_{1\le k\le m}\binom{m}{k}
	\mathbb{E}(T_n^{m-k})\sum_{1\le j<n}
	\sum_{\substack{S\in\WD\\ S=W\cup D}}
	\binom{n}{j}\pi_j^{(W)}\pi_{n-j}^{(D)}\mu_{j,k}\\
	&=O\left(\rho^{-n}\sum_{1\le k\le m}\rho^{-n(m-k)}
	\sum_{\substack{S\in\WD\\ S=W\cup D}}
	\sum_{1\le j<n}\binom{n}{j}\pi_j^{(W)}\pi_{n-j}^{(D)}
	\rho^{-jk}\right)\\
	&=O\left(\sum_{1\le k\le m}\rho^{n(m-k+1)}
	\sum_{\substack{S\in\WD\\ S=W\cup D}}
	\left(\rho^{-k}\sum_{j\in W}p_j
	+\sum_{\ell\in D}p_\ell\right)^n\right)\\
	&=O\left(\left(\frac{\tilde{\rho}}{\rho}\right)^n
	\rho^{-nm}\right),
\end{align*}
where
\[
	\tilde{\rho}=\max\left\{\sum_{j\in W}p_j
	+\rho\sum_{j\in D} p_j\ :\ S=W\cup D\in\WD\right\}.
\]
Since $\tilde{\rho}<\rho$, we see that $(\tilde{\rho}/\rho)^n$ is
exponentially smaller than 1. Thus by substituting the above estimate
and Lemma \ref{ref-toll} into (\ref{rec-mth-mom-2}), we prove
\eqref{munm-bd2}.\qed

The estimates \eqref{munm-bd2} imply the convergence in distribution
of $X_n$ by the method of moment. This proves
Theorem~\ref{main-thm-exp}.

\section{Applications}\label{examples}

We apply in this section the previous theorems to a few concrete
cases.

\paragraph{\CTLS.} For leader selection by coin tossing, the set of
hands is given by $\{\mathscr{H}_1,\mathscr{H}_2\}$, where
$\mathscr{H}_1$ corresponds to head, and $\mathscr{H}_2$ to tail. The
only one WOD set is $\{\mathscr{H}_1,\mathscr{H}_2\}$, where
$\mathscr{H}_1$ is the winning hand and $\mathscr{H}_2$ the losing
hand. The remaining subsets of hands all lead to ties. Thus \CTLS\ is
a log-game and both Theorems~\ref{main-thm1-log} and
\ref{main-thm2-log} apply.

For instance, in the case of an unbiased coin (i.e.,
$p_1=p_2=\frac12$), we obtain $\varpi_n=1-2^{1-n}$,
and for (\ref{In})
\[
	\mathbb{P}(I_n=j)
	=\begin{cases}
	    \displaystyle 2^{1-n},
		&\text{if}\ j=n;\\
		{\displaystyle 2^{-n}\binom{n}{j}},
		&\text{if}\ 1\le j<n.
	\end{cases}
\]
Thus the functional equation (\ref{gen-func-eq}) satisfied by the
Poisson generating functions of the moments has the form
\[
    \tilde{f}(z)=\left(1+e^{-z/2}\right)
    \tilde{f}\left(\frac{z}{2}\right)+\tilde{g}(z),
\]
with $\tilde{f}(0)=\tilde{g}(0)=0$,
which can be further simplified by considering
\[
	\hat{f}(z):=\frac{\tilde{f}(z)}{1-e^{-z}},
	\qquad\hat{g}(z):=\frac{\tilde{g}(z)}{1-e^{-z}},
\]
so that (by $\hat{g}(0)=0$)
\[
    \hat{f}(z)=\hat{f}\left(\frac{z}{2}\right)+\hat{g}(z)
    = \sum_{j\ge0}\hat{g}\left(\frac z{2^j}\right).
\]
From this we can derive an explicit expression for
the Fourier coefficients in Proposition \ref{asym-gen-func-eq}
(compare Remark \ref{exp-four}); see \cite{Pr} for details.

Furthermore, the Poisson generating function for the
distribution function satisfies (see (\ref{func-eq-Al}))
\[
	\tilde{A}_{\ell+1}(z)
	=\left(1+e^{-z/2}\right)\tilde{A}_{\ell}
	\left(\frac{z}{2}\right).
\]
Then by Remark \ref{exp-ll}, the sum on the right-hand side in
Theorem \ref{main-thm2-log} indeed simplifies to the expression in
(\ref{unbiased-ll}).

In addition, our results also apply to the case of an unbiased coin
($p_1\ne p_2$) and one then recovers the results from \cite{JaSz}.
In this case, the sum on the right-hand side in
Theorem \ref{main-thm2-log} can be written more elegantly as an
integral with respect to a suitable defined point measure; see
\cite{JaSz} for details.

\paragraph{\RPSLS.}
For the classical Janken (RPS) game, the set of hands is
$\{\mathscr{H}_1,\mathscr{H}_2,\mathscr{H}_3\}$ and all subsets of
cardinality two are WOD sets, i.e., $\{\mathscr{H}_1,\mathscr{H}_2\},
\{\mathscr{H}_1,\mathscr{H}_3\}$ and $\{\mathscr{H}_2,
\mathscr{H}_3\}$, all other subsets leading to ties. Thus \RPSLS\ is
an exp-game and Theorem~\ref{main-thm-exp} applies.

For instance, for $p_1=p_2=p_3=\frac13$, we have $\rho=\frac23$ and $\nu=3$ and one recovers the main result from \cite{MaUe}; see also Section \ref{intro}.

\paragraph{Other three-hand games.} Apart from \RPSLS\ many other
games with three hands are possible. We content ourselves here with a
short discussion of games defined on a dominance graph with
$p_1=p_2=p_3=\frac13$. All five connected dominance graphs are given
as follows.
\newpage
%\vspace*{0.2cm}
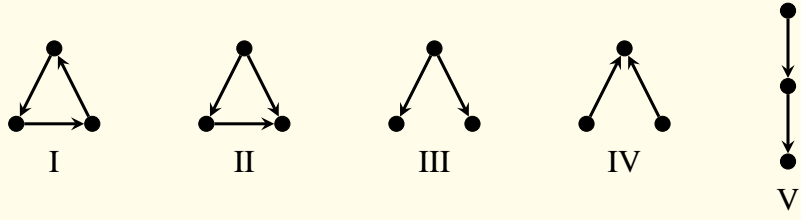
\begin{center}
\begin{tikzpicture}
\draw (0cm,0cm) node[circle,inner sep=2,fill,draw] (1) {};
\draw (0.5cm,1cm) node[circle,inner sep=2,fill,draw] (2) {};
\draw (1cm,0cm) node[circle,inner sep=2,fill,draw] (3) {};
\draw (0.5cm,-0.5cm) node (4) {I};

\draw[-stealth,line width=0.4mm] (1) -- (3);
\draw[-stealth,line width=0.4mm] (3) -- (2);
\draw[-stealth,line width=0.4mm] (2) -- (1);

\draw (2.5cm,0cm) node[circle,inner sep=2,fill,draw] (1) {};
\draw (3cm,1cm) node[circle,inner sep=2,fill,draw] (2) {};
\draw (3.5cm,0cm) node[circle,inner sep=2,fill,draw] (3) {};
\draw (3cm,-0.5cm) node (4) {II};

\draw[-stealth,line width=0.4mm] (2) -- (1);
\draw[-stealth,line width=0.4mm] (2) -- (3);
\draw[-stealth,line width=0.4mm] (1) -- (3);

\draw (5cm,0cm) node[circle,inner sep=2,fill,draw] (1) {};
\draw (5.5cm,1cm) node[circle,inner sep=2,fill,draw] (2) {};
\draw (6cm,0cm) node[circle,inner sep=2,fill,draw] (3) {};
\draw (5.5cm,-0.5cm) node (4) {III};

\draw[-stealth,line width=0.4mm] (2) -- (1);
\draw[-stealth,line width=0.4mm] (2) -- (3);

\draw (7.5cm,0cm) node[circle,inner sep=2,fill,draw] (1) {};
\draw (8cm,1cm) node[circle,inner sep=2,fill,draw] (2) {};
\draw (8.5cm,0cm) node[circle,inner sep=2,fill,draw] (3) {};
\draw (8cm,-0.5cm) node (4) {IV};

\draw[-stealth,line width=0.4mm] (1) -- (2);
\draw[-stealth,line width=0.4mm] (3) -- (2);

\draw (10.15cm,-0.5cm) node[circle,inner sep=2,fill,draw] (1) {};
\draw (10.15cm,0.5cm) node[circle,inner sep=2,fill,draw] (2) {};
\draw (10.15cm,1.5cm) node[circle,inner sep=2,fill,draw] (3) {};
\draw (10.15cm,-1cm) node (4) {V};

\draw[-stealth,line width=0.4mm] (2) -- (1);
\draw[-stealth,line width=0.4mm] (3) -- (2);
\end{tikzpicture}
\end{center}
The games based on these graphs are played in a natural way. Only the
\GJLS\ arising from Graph V requires some more explanation. Assume
that the hands are $\{\mathscr{H}_1, \mathscr{H}_2, \mathscr{H}_3\}$,
where ${\mathscr{H}_i}$ corresponds to the $i$-th node from the top
to the bottom on Graph V. Then all WOD sets are $\{\mathscr{H}_1,
\mathscr{H}_2, \mathscr{H}_3\}$ (with $\mathscr{H}_1$ the winning
hand), $\{\mathscr{H}_1, \mathscr{H}_2\}$ (with $\mathscr{H}_1$ the
winning hand) and $\{\mathscr{H}_2, \mathscr{H}_3\}$ (with
$\mathscr{H}_2$ the winnings hand). In particular, note that we
define $\{\mathscr{H}_1, \mathscr{H}_3\}$ as a tie (no transitivity);
for a more general definition of games arising from general dominance
graphs, see the paragraph on Janken games around the world below.

Note that the \GJLS\ defined on Graph I corresponds to \RPSLS\ and is
the only exp-game; games on all other graphs are log-games.

For instance, the \GJLS\ defined on Graph II has $\varpi_n=
1-3^{1-n}$, and (\ref{In}) has the form
\[
	\mathbb{P}(I_n=j)
	=\begin{cases}
	    3^{1-n},
		&\text{if}\ j=n;\\
		{\displaystyle 3^{-n}\binom{n}{j}}
		\left(2^{n-j}+1\right),
		&\text{if}\ 1\le j<n.
	\end{cases}
\]
The mean and the variance are then described by applying Theorem
\ref{main-thm1-log}. For the limit law, observe that
(\ref{func-eq-Al}) becomes
\[
	\tilde{A}_{\ell+1}(z)
	=\left(1+e^{-z/3}+e^{-2z/3}\right)
	\tilde{A}_{\ell}\left(\frac{z}{3}\right).
\]
Then the explicit expression in Remark \ref{exp-ll} specializes to
\[
	\tilde{A}_{\ell}(z)
	=\frac{z}{3^{\ell}}e^{-z/3^{\ell}}
	\prod_{1\le j\le \ell}\left(1+e^{-z/3^{j}}+e^{-2z/3^j}\right),
\]
and one then obtains for the sum on the right-hand side in Theorem
\ref{main-thm2-log}
\begin{align*}
	\sum_{j\ge 0}\tilde{R}_j(3^j z)
	&=ze^{-z}+ze^{-z}\sum_{k\ge 0}
	\left(e^{-3^{k}z}+e^{-2\cdot 3^{k}z}\right)
	\prod_{0\le j<k}
	\left(1+e^{-3^{j}z}+e^{-2\cdot 3^{j}z}\right)\\
	&= \frac{z}{e^{z}-1}.
\end{align*}
Note that this gives the same result as for \CTLS\ with an unbiased coin, with the only difference that $2$ is replaced by $3$ in (\ref{unbiased-ll}); for a generalization with $2$ replaced by any number $m$ see the next paragraph below.

Consider now the games defined on Graphs III and IV, which can be
seen to correspond to \CTLS\ with a biased coin whose probability of
head is either $\frac13$ (Graph III) or $\frac23$ (Graph IV). For
example, for the game defined on Graph III, we have $\varpi_n=1-
(2^n+1) 3^{-n}$ and then (\ref{In}) becomes
\[
	\mathbb{P}(I_n=j)
	=\begin{cases}{\displaystyle (2^n+1)3^{-n}},
	&\text{if}\ j=n;\\
	{\displaystyle \frac{2^{n-j}}{3^n}\binom{n}{j}},
	&\text{if}\ 1\le j<n.
	\end{cases}
\]

Finally, the \GJLS\ defined on Graph V is also equivalent to \CTLS\
with a biased coin, and $\varpi_n$ and (\ref{In}) are identical to
those of the \GJLS\ on Graph III.

In summary, there are exactly five different three-hand games defined
on connected dominance graphs, where hands are chosen uniformly at
random. One of them is an exp-game and all others are log-games.
Moreover, the exp-game is \RPSLS\ and three of the four log-games
correspond to \CTLS\ with a biased coin. Finally, the fourth log-game is
a natural extension of $\CTLS$ from two hands to three hands.

\paragraph{Janken games on acyclic cliques and \CTLS.}
We consider Janken games defined on a directed acyclic clique.
Note that such a clique contains exactly one node of out-degree
$i$ for $i=0, \dots,m-1$; see the following graphs for $m=2,\dots,6$.
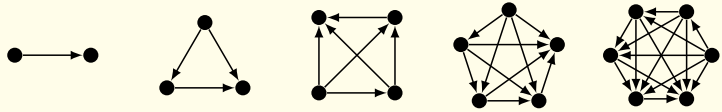
\begin{center}
\begin{tikzpicture}[scale=1,transform shape]
\node[] (P1) at (0,0) {
\begin{tikzpicture}
\node [circle, fill=black,inner sep=2pt]  (n1) at (0,0) {};
\node [circle, fill=black,inner sep=2pt, xshift = 1cm]
(n2) at (n1) {};

\draw[-latex,line width=0.2mm] (n1) -- (n2);
\end{tikzpicture}
};

\node[] (P2) at (2,0) {
\begin{tikzpicture}
\node [circle, fill=black,inner sep=2pt]  (n1) at (0,0) {};
\node [circle, fill=black,inner sep=2pt, xshift = 1cm]
(n2) at (n1) {};
\node [circle, fill=black,inner sep=2pt, xshift = 0.5cm,
yshift=0.866cm]  (n3) at (n1) {};

\draw[-latex,line width=0.2mm] (n1) -- (n2);
\draw[-latex,line width=0.2mm] (n3) -- (n2);
\draw[-latex,line width=0.2mm] (n3) -- (n1);
\end{tikzpicture}
};

\node[] (P3) at (4,0) {
\begin{tikzpicture}
\node [circle, fill=black,inner sep=2pt]  (n1) at (0,0) {};
\node [circle, fill=black,inner sep=2pt, xshift = 1cm]
(n2) at (n1) {};
\node [circle, fill=black,inner sep=2pt, xshift = 0cm, yshift=1cm]
(n3) at (n1) {};
\node [circle, fill=black,inner sep=2pt, xshift = 1cm, yshift=1cm]
(n4) at (n1) {};

\draw[-latex,line width=0.2mm] (n1) -- (n2);
\draw[-latex,line width=0.2mm] (n2) -- (n3);
\draw[-latex,line width=0.2mm] (n2) -- (n4);
\draw[-latex,line width=0.2mm] (n4) -- (n3);
\draw[-latex,line width=0.2mm] (n1) -- (n4);
\draw[-latex,line width=0.2mm] (n1) -- (n3);
\end{tikzpicture}
};

\node[] (P4) at (6,0) {
\begin{tikzpicture}[]
\node [circle, fill=black,inner sep=2pt]
(n1) at (0,0.66666) {};
\node [circle, fill=black,inner sep=2pt]
(n2) at (-0.63403, 0.206) {};
\node [circle, fill=black,inner sep=2pt]
(n3) at (-0.39185, -0.53934) {};
\node [circle, fill=black,inner sep=2pt]
(n4) at (0.39185, -0.53934) {};
\node [circle, fill=black,inner sep=2pt]
(n5) at (0.63403, 0.206) {};

\draw[-angle 45,-latex,line width=0.2mm] (n1) -- (n2);
\draw[-angle 45,-latex,line width=0.2mm] (n2) -- (n3);
\draw[-angle 45,-latex,line width=0.2mm] (n3) -- (n4);
\draw[-angle 45,-latex,line width=0.2mm] (n4) -- (n5);
\draw[-angle 45,-latex,line width=0.2mm] (n1) -- (n5);
\draw[-angle 45,-latex,line width=0.2mm] (n1) -- (n3);
\draw[-angle 45,-latex,line width=0.2mm] (n1) -- (n4);
\draw[-angle 45,-latex,line width=0.2mm] (n2) -- (n4);
\draw[-angle 45,-latex,line width=0.2mm] (n2) -- (n5);
\draw[-angle 45,-latex,line width=0.2mm] (n3) -- (n5);
\end{tikzpicture}
};

\node[] (P5) at (8,0) {
\begin{tikzpicture}[]
\node [circle, fill=black,inner sep=2pt]
(n1) at (0.66666, 0) {};
\node [circle, fill=black,inner sep=2pt]
(n2) at (0.33333, 0.57734) {};
\node [circle, fill=black,inner sep=2pt]
(n3) at (-0.33333, 0.57734) {};
\node [circle, fill=black,inner sep=2pt]
(n4) at (-0.66666, 0) {};
\node [circle, fill=black,inner sep=2pt]
(n5) at (-0.33333, -0.57734) {};
\node [circle, fill=black,inner sep=2pt]
(n6) at (0.33333, -0.57734) {};
\draw[-angle 45,-latex,line width=0.2mm] (n1) -- (n2);
\draw[-angle 45,-latex,line width=0.2mm] (n2) -- (n3);
\draw[-angle 45,-latex,line width=0.2mm] (n3) -- (n4);
\draw[-angle 45,-latex,line width=0.2mm] (n4) -- (n5);
\draw[-angle 45,-latex,line width=0.2mm] (n5) -- (n6);
\draw[-angle 45,-latex,line width=0.2mm] (n1) -- (n6);
\draw[-angle 45,-latex,line width=0.2mm] (n1) -- (n3);
\draw[-angle 45,-latex,line width=0.2mm] (n1) -- (n4);
\draw[-angle 45,-latex,line width=0.2mm] (n1) -- (n5);
\draw[-angle 45,-latex,line width=0.2mm] (n2) -- (n4);
\draw[-angle 45,-latex,line width=0.2mm] (n2) -- (n5);
\draw[-angle 45,-latex,line width=0.2mm] (n2) -- (n6);
\draw[-angle 45,-latex,line width=0.2mm] (n3) -- (n5);
\draw[-angle 45,-latex,line width=0.2mm] (n3) -- (n6);
\draw[-angle 45,-latex,line width=0.2mm] (n4) -- (n6);
\end{tikzpicture}
};
\end{tikzpicture}	
\end{center}

The \GJLS\ (played in the natural way) on such a directed acyclic graph is a natural extension of
\CTLS\ with an unbiased coin, and the \GJLS\ from the previous
paragraph on Graph II. The functional equation for the moments becomes
\[
	\tilde{f}(z)
	=\left(1+e^{-z/m}+e^{-2z/m}+\cdots+e^{-(m-1)z/m}\right)
	\tilde{f}\left(\frac{z}{m}\right)+\tilde{g}(z).
\]
Thus the same normalization technique as that used for \CTLS\ with an
unbiased coin applies, and the Fourier coefficients in Theorem
\ref{main-thm1-log} can be made more explicit.

Moreover, for the
asymptotic distribution, we have
\[
	\tilde{A}_{\ell+1}(z)
	=\left(1+e^{-z/m}+e^{-2z/m}+\cdots+e^{-(m-1)z/m}\right)
	\tilde{A}_{\ell}\left(\frac{z}{m}\right),
\]
which gives again the same form as (\ref{unbiased-ll}) but
with $2$ there replaced by $m$.

\paragraph{Regular tournament Janken games.}\label{reg-Janken} A
\emph{tournament} is a directed graph (digraph) obtained by assigning
a direction for every edge in a complete graph. A
\emph{regular tournament} of $2m+1$ nodes is a tournament in which
every node dominates exactly $m$ other nodes and is dominated by the
remaining $m$ nodes.

Typical examples of regular tournaments are the dominance graphs of
\RPSLS\ and the rock-paper-scissors-Spock-lizard game. Also the
dominance graph of the rock-paper-scissors-well game is a
subtournament of the rock-paper-scissors-Spock-lizard game. Note that
up to isomorphism there is only one regular tournament of $5$ nodes,
two regular tournaments of $7$ nodes, and the number increases very
rapidly with the number of nodes. The case of regular tournaments on
infinitely many nodes also makes sense and will be briefly discussed
in Section~\ref{con}.

We consider a class of Janken games arising from a regular tournament
generalizing the RPS game. The hands $\{\mathscr{H}_1, \ldots,
\mathscr{H}_{2m+1}\}$ correspond to the nodes of the regular
tournament. For instance, the hands can be ordered cyclically such
that $\mathscr{H}_i$ dominates $\mathscr{H}_j$ for $j-i\equiv \{1,
\ldots, m\} \bmod (2m+1)$. The game is played as follows: if the $n$
players choose the hands $S=\{\mathscr{H}_{i_1}, \ldots,
\mathscr{H}_{i_{\ell}}\}$, then the game is in a tie if and only if
either of the following condition holds:
\begin{itemize}
	\item $S$ is a singleton, or
	\item the induced subgraph of $S$ contains a cycle.
\end{itemize}
Otherwise the game is in a WOD situation.

Since by definition, the set of all hands is a tie, this is an
exp-game. Thus the cost of this Janken game is described by Theorem
\ref{main-thm-exp}. For instance, if $p_i=1/(2m+1)$ for $1\le i\le
2m+1$, then $\rho=(m+1)/(2m+1)$ and $\nu=2m+1$. On the other hand, if
we do not count rounds in which the induced subgraph of $S$ contains
a cycle, then the number of steps needed is predicted by Theorem
\ref{main-thm1-log} and Theorem \ref{main-thm2-log}; see Section
\ref{ext} where counting without ties is considered in more
generality.

\paragraph{Some other Janken games around the world.} Many variants of
RPS are played all over the world. Examples include (i) the
rock-paper-scissors-well game in Germany; (ii) the
bird-stone-revolver-plank-water game in Malaysia, and (iii) the
god-chicken-rifle-termite-fox game in Guangdong province (in China);
see the following figure for an illustration of (ii) and (iii). See
also \cite{MaUe}.
\begin{figure}
\begin{center}
\begin{tikzpicture}
\draw (-7cm,4.31cm) node (11) {B};
\draw (-9.27cm,2.66cm) node (12) {S};
\draw (-8.4cm,0cm) node (13) {R};
\draw (-5.6cm,0cm) node (14) {P};
\draw (-4.73cm,2.66cm) node (15) {W};

\draw[-latex,line width=0.4mm] (11) -- (15);
\draw[-latex,line width=0.4mm] (12) -- (11);
\draw[-latex,line width=0.4mm] (12) -- (14);
\draw[-latex,line width=0.4mm] (13) -- (11);
\draw[-latex,line width=0.4mm] (13) -- (12);
\draw[-latex,line width=0.4mm] (13) -- (14);
\draw[-latex,line width=0.4mm] (14) -- (11);
\draw[-latex,line width=0.4mm] (14) -- (15);
\draw[-latex,line width=0.4mm] (15) -- (12);
\draw[-latex,line width=0.4mm] (15) -- (13);

\draw (0cm,4.31cm) node (1) {G};
\draw (-2.27cm,2.66cm) node (2) {C};
\draw (-1.4cm,0cm) node (3) {R};
\draw (1.4cm,0cm) node (4) {T};
\draw (2.27cm,2.66cm) node (5) {F};

\draw[-latex,line width=0.4mm] (1) -- (2);
\draw[-latex,line width=0.4mm] (1) -- (3);
\draw[-latex,line width=0.4mm] (2) -- (4);
\draw[-latex,line width=0.4mm] (3) -- (2);
\draw[-latex,line width=0.4mm] (3) -- (5);
\draw[-latex,line width=0.4mm] (4) -- (1);
\draw[-latex,line width=0.4mm] (5) -- (2);

\draw (-7cm,-0.8cm) node (1){\small Bird, Stone, Revolver,
Plank, Water};
\draw (-7cm,-1.5cm) node (3) {\small (MALAYSIA)};
\draw (0.1cm,-0.8cm) node (4){\small God, Chicken, Rifle, Termite, Fox};
\draw (0.1cm,-1.5cm) node (5) {\small (Guangdong, CHINA)};
\end{tikzpicture}
\end{center}
\vspace*{-.5cm}
\end{figure}

All the three games (i)--(iii) are mainly played as two-person games.
They can be generalized to $n$-person games. More precisely, assume
that the hands are given by $\{\mathscr{H}_1, \ldots,
\mathscr{H}_m\}$. If the $n$ players choose the hands
$S=\{\mathscr{H}_{i_1}, \ldots,\mathscr{H}_{i_{\ell}}\}$, then we
consider the subgraph induced by these hands. The game is in a tie if
and only if either the subgraph consists only of isolated nodes or
there is a cycle in one of the connected components of the subgraph.
(Note that in contrast to the regular tournament Janken game, induced
subgraphs are now not necessarily connected). Thus, if $S$ is a WOD
set (i.e., not a tie), then the induced subgraph has at least one
node with only outgoing edges. All these nodes together with the
isolated nodes form the set of winning hands; the remaining hands are
the losing hands.

As an example, assume that in the
god-chicken-rifle-termite-fox game, the players choose the hands
$\{G,C,F\}$. Then this set is a WOD set where $\{G,F\}$ are the
winning hands and $\{C\}$ is the losing hand.

The games (i)--(iii) are all of exponential type since there are cycles
in the dominance graphs. Thus the number of rounds is described by
Theorem \ref{main-thm-exp}. For instance, if hands are chosen
uniformly at random, then one obtains the following game indices:

\vspace*{0.05cm}
\begin{center}
\setlength\extrarowheight{2.5pt}
\begin{tabular}{|l||c|c|}
\hline
\multicolumn{1}{|c||}{\GJLS} & $\rho$ & $\nu$ \\
\hline
\hline
(i) rock-paper-scissors-well & $\frac34$ & $2$ \\
(ii) rock-stone-revolver-plank-water & $\frac45$ & $1$ \\
(iii) god-chicken-rifle-termite-fox & $\frac45$ & $3$ \\[2.5pt]
\hline
\end{tabular}
\end{center}

\paragraph{Games on a circulant payoff matrix.}
These games where introduced in \cite{SuOs} for selecting a leader in
a distributed system. In contrast to all the previous games, they are
not defined via an underlying dominance graph. We first recall their
definition.

The games are played with $2m+1$ hands $\{\mathscr{H}_1, \ldots,
\mathscr{H}_{2m+1}\}$, where for $1\le i,j\le 2m+1$ with $i\ne j$,
the hand $\mathscr{H}_i$ gets a payoff of $2^{m+g(i,j)}$ from
$\mathscr{H}_j$. Here $i-j\equiv g(i,j) \bmod (2m+1)$ with $-m\le
g(i,j)\le m$. If the players choose the hands $S
=\{\mathscr{H}_{i_1}, \ldots, \mathscr{H}_{i_{\ell}}\}$, then the
total gain of $\mathscr{H}_{i_j}$ is defined as
\[
    G_{i_{j},S}=\sum_{\mathscr{H}_{i_k}\in S}2^{m+g(i_j,i_k)}.
\]
Denote by $S_M$ the subset of hands from $S$ with the maximal total
gain. Then $S$ is a WOD set with winning hands $S_M$ and
losing hands $S\setminus S_M$ if and only if $\#S_M<\#S$.

Examples of a three-hand game and a five-hand game defined on
circulant payoff matrices are given in Figure \ref{3-fig}. Note that
$m=1$ corresponds to \RPSLS. On the other hand, the game with $m=2$
is different from the rock-paper-scissors-Spock-lizard game since
there are more WOD sets in the former. This is also the reason why
the five-hand circulant payoff matrix game was described as more
efficient than the five-hand regular tournament Janken game in
\cite{SuOs}.

\begin{table}[!ht]
\begin{center}
\begin{tabular}{c|ccc}
    hands &
	rock & paper & scissors \\ \hline
    rock & $2^1$ & $2^0$ & $2^2$ \\
    paper & $2^2$ & $2^1$& $2^0$ \\
    scissors & $2^0$ & $2^2$ & $2^1$
\end{tabular}\quad
\begin{tabular}{c|ccccc}
    hands & $\mathscr{H}_1$
	& $\mathscr{H}_2$ & $\mathscr{H}_3$
	& $\mathscr{H}_4$ & $\mathscr{H}_5$ \\ \hline
    $\mathscr{H}_1$ & $2^2$ & $2^1$ & $2^0$ & $2^4$ & $2^3$ \\
    $\mathscr{H}_2$ & $2^3$ & $2^2$ & $2^1$ & $2^0$ & $2^4$ \\
    $\mathscr{H}_3$ & $2^4$ & $2^3$ & $2^2$ & $2^1$ & $2^0$ \\
    $\mathscr{H}_4$ & $2^0$ & $2^4$ & $2^3$ & $2^2$ & $2^1$ \\
    $\mathscr{H}_5$ & $2^1$ & $2^0$ & $2^4$ & $2^3$ & $2^2$
\end{tabular}
\end{center}
\vspace*{-.4cm}
\caption{\emph{Three-hand (left) and five-hand (right) circulant
payoff matrix game.}}\label{3-fig}
\end{table}

The set of all hands in a circulant payoff matrix game is obviously a
tie because each hand has the same total gain. Thus, circulant
payoff matrix games are of exponential type and their complexity
are described by Theorem \ref{main-thm-exp}. For instance, if
$p_i=1/(2m+1)$ for $1\le i\le 2m+1$, then $\rho=2m/(2m+1)$ and
$\nu=2m+1$. This generalizes the observation made in \cite{SuOs},
namely, \emph{$(2m+1)$-hand circulant payoff matrix games are more
efficient than $(2m+1)$-hand regular tournament games.}

\section{Extensions and Other Gap Theorems}\label{ext}

Another equally important cost measure for \GJLS\ is the total number
of hands used until a leader is selected, which corresponds to the
number of times the random variate is generated if the hands are
generated by random mechanisms. We show that instead of the log-exp
complexity change, there is a change from linear to exponential for
the expected cost; also the limit law exists in either case (log or
exp).

On the other hand, we saw, from the above analysis, that ties play a
crucial role in distinguishing between log-games and exp-games. We
examine in this section the contribution of ties in slightly more
detail. More precisely, we discuss the number of rounds until a
leader is selected when ties are ignored.

\subsection{The total number of hands}

We consider here the total number of hands $Y_n$ used by all players
before finding a leader by \GJLS. It turns out that $Y_n$
exhibits a scale change from linear ($\rho=1$) to exponential
($\rho<1$).

We begin with the case $\rho=1$ for which we have the recurrence
(cf. \eqref{dis-rec-1})
\[
    Y_n\stackrel{d}{=}Y_{I_n}+n,\qquad (n\ge 2),
\]
where $Y_n,I_n$ are independent, $Y_1=0$ and $I_n$ is defined
in \eqref{In}.

While the asymptotic distribution of $X_n$ is
dictated by periodic oscillations, the distribution of $Y_n$ follows
a central limit theorem.
\begin{thm}[Log-games]
The total number of hands $Y_n$ used by \GJLS\ to find a leader
in log-games ($\rho=1$) is asymptotically normally distributed
\[
	\frac{Y_n-\mathbb{E}(Y_n)}{\sqrt{\mathbb{V}(Y_n)}}
	\stackrel{d}{\longrightarrow} \mathscr{N}(0,1),
\]
where $\mathscr{N}(0,1)$ denotes the standard normal distribution. Furthermore, the mean and the variance satisfy
\begin{align}\label{EV-Yn}
\begin{split}
	\mathbb{E}(Y_n)
	&=\frac{1}{1-\alpha}\,n+P_3(\log_{1/\alpha} n)+o(1),\\
    \mathbb{V}(Y_n)&=\frac{\alpha}{(1-\alpha)^2}\,n+
	P_4(\log_{1/\alpha}n)\log n +O(1),
\end{split}	
\end{align}
where $P_3(z),P_4(z)$ are one-periodic functions.
\end{thm}
\begin{proof}(Sketch)
As in Section~\ref{log-case}, we consider
\[
	\tilde{f}_1(z)
	:=e^{-z}\sum_{n\ge 1}\mathbb{E}(Y_n)\frac{z^n}{n!}
	\qquad\text{and}\qquad\tilde{f}_2(z)
	:=e^{-z}\sum_{n\ge 1}\mathbb{E}(Y_n^2)\frac{z^n}{n!},
\]
which, by a similar computation as in Section \ref{log-case},
satisfy the functional equations
\begin{align*}
    \tilde{f}_1(z)&=\tilde{f}_1(\alpha z)+z+\tilde{\phi}_1(z)\\
	\tilde{f}_2(z)&=\tilde{f}_2(\alpha z)
	+2z\tilde{f}_1(\alpha z)
	+2\alpha z\tilde{f}_1'(\alpha z)
	+z^2+z+\tilde{\phi}_2(z),
\end{align*}
where $\tilde{\phi}_1(z)$ and $\tilde{\phi}_2(z)$ are finite
sums of exponentially small terms.

The proof of \eqref{EV-Yn} then follows the same line of arguments
as that of Theorem~\ref{main-thm1-log} except for the variance for
which we need to consider the Poissonized variance (see \cite{HwFuZa})
\[
    \tilde{V}(z):=\tilde{f}_2(z)
	-\tilde{f}_1(z)^2-z\tilde{f}'_1(z)^2,
\]
which then satisfies the equation
\[
	\tilde{V}(z)=\tilde{V}(\alpha z)+\alpha(1-\alpha)z
	\tilde{f}'_1(\alpha z)^2+\tilde{\phi}_3(z),
\]
where $\tilde{\phi}_3(z)$ consists of exponentially small terms.

Finally, the central limit theorem can be proved either by the method of
moments using the shifting-the-mean technique (see \cite{CH03}) or by
the contraction method \cite{NeRu}; details are omitted here.
\end{proof}

Note that for \CTLS\ with an unbiased coin, it was proved in
\cite{Pr} that $\mathbb{E}(Y_n)=2n$ ($\alpha=\frac12$). This is a
rather exceptional result which does not seem to hold in general.
Furthermore, in this case, $P_4(z)\equiv 0$.

We now consider exp-games ($\rho<1$). Here the total number of
hands used satisfies the distributional recurrence
\[
    Y_n\stackrel{d}{=}Y_{J_n}+nT_n,\qquad (n\ge 2),
\]
where $Y_n,J_n,T_n$ are independent, $T_n$ is a geometrically
distributed random variable with parameter $\varpi_n$,
$Y_1=0$ and $J_n$ is defined in \eqref{Jn}.

Then the same approach used in Section \ref{exp-case} applies and we
obtain the same exponential limit law for $Y_n$.
\begin{thm}[Exp-games]
The total number of hands $Y_n$ used to find a leader by
\GJLS\ in exp-games ($\rho<1$) satisfies
\[
    \frac{\nu\rho^n}{n}Y_n
	\stackrel{d}{\longrightarrow}{\rm Exp}(1),
\]
where ${\rm Exp}(1)$ denotes an exponentially distributed random
variable with mean one. In addition, we have convergence of all
moments. In particular, the mean satisfies
\[
    \mathbb{E}(Y_n)\sim\frac{n}{\nu\rho^n}.
\]
\end{thm}

Overall, we see that the total number of hands used by
\GJLS\ exhibits a sharp scale change from linear to exponential.

\subsection{Counting without ties}

Denote by $Z_n$ the number of rounds used by \GJLS\ until a leader is
selected where ties are ignored. Then $Z_n$ satisfies the same
recurrence as (\ref{dis-rec-2}) with $T_n$ there replaced by $1$,
namely,
\[
    Z_n\stackrel{d}{=}Z_{J_n}+1,\qquad (n\ge 2),
\]
where $Z_n, J_n$ are independent, $Z_1=0$ and $J_n$ is defined in
\eqref{Jn}. It turns out that in such a case the expected cost is
always logarithmic, independent of $\rho$. For simplicity, we
consider only the mean.

We introduce first some notations. Denote by $S_1,\ldots S_{\nu}$
the WOD sets for which the maximum is attained in (\ref{def-rho}).
Furthermore, define
\[
	\alpha_{\ell}:=\sum_{\mathscr{H}_j\in W_{\ell}}p_j,
	\qquad (1\le\ell\le\nu),
\]
where $W_{\ell}$ is the set of winning hands belonging to $S_{\ell}$.

By the same arguments used in Section \ref{log-case},
we obtain the functional equation
\[
	\tilde{f}_1(z)=\frac{1}{\nu}\sum_{1\le \ell\le \nu}
	\tilde{f}_1\left(\frac{\alpha_{\ell}}{\rho}\,z\right)
	+1+\tilde{\phi}_1(z),
\]
with $\tilde{f}_1(0)=\tilde{\phi}_1(0)=0$, where
\[
    \tilde{f}_1(z):=e^{-z}\sum_{n\ge 1}
	\mathbb{E}(Z_n)\frac{z^n}{n!},
\]
and $\tilde{\phi}_1(z)$ is a finite sum of exponentially small terms.

Asymptotics of $\mathbb{E}(Z_n)$ then follows from the same
Mellin and de-Poissonization arguments we used above. The
major difference here is that the Mellin transform of
$\tilde{f}_1(z)$ has now the denominator
\[
	1-\frac{1}{\nu}\sum_{1\le \ell\le \nu}
	\left(\frac{\rho}{\alpha_{\ell}}\right)^s,
\]
instead of $1-\alpha^{-s}$ when $\rho=1$; see (\ref{meltr-F}).

For the inverse Mellin transform, one needs to clarify the set of zeros of this function for which much has been known; see
\cite{FFH86,FlRoVa} for more information. First, it is easy to see
that all zeros must satisfy $\Re(s)>0$. Moreover, it is well-known
that $s=0$ is the only zero on the vertical line $\Re(s)=0$ if and
only if at least one of the ratios $\log(\rho/\alpha_j)/
\log(\rho/\alpha_k)$ is irrational. Note that if the latter does not
hold, then there exits an $r>1$ such that $\rho/\alpha_j
=r^{\kappa_j}$ for positive integers $\kappa_j$, and there are
infinitely many zeros on $\Re(s)=0$ that are equally spaced along
this line; for these and related properties see, e.g., \cite{FlRoVa}
and references therein.

Using this and the approach from Section \ref{log-case}, we obtain the following result.
\begin{thm}
The expected number of rounds $Z_n$, counted without ties, used by
\GJLS\ to find a leader satisfies
\[
    \mathbb{E}(Z_n)= h_\nu \log n +P(\log_r n) + o(1),
\]
where
\[
	h_\nu := \frac1{\frac{1}{\nu}\sum_{1\le \ell\le \nu}
	\log\left(\frac{\rho}{\alpha_{\ell}}\right)},
\]
and $P(z)$ is a constant if at least one of the ratios
$\log(\rho/\alpha_j)/\log(\rho/\alpha_k)$ is
irrational, and is a one-periodic function otherwise.
\end{thm}
Thus ties dominate in an exp-game.

\section{A concluding remark}\label{con}

In this paper, we discussed leader selection based on generalized
Janken games. Our framework contains as special cases the classical
leader selection procedure using coin-tossing, which is widely used
in computer science and related areas, and the popular two-person RPS
game (and its variants), which is played in many countries and is of
importance in game theory, biology and physics. We showed that leader
selection based on the latter as well as many other previous examples
of Janken games exhibit only dichotomous behaviors: either very
efficient of log-order or very laborious of exp-order.

We conclude this paper with a remark on an infinite-hand Janken game
generalizing the regular tournament Janken game introduced in Section
\ref{reg-Janken}.

Assume that the $n$ players choose points uniformly at random on the
surface of the unit sphere in the ${\ell}$-space. Then the
probability that all points lie on the same hemisphere \cite{Wen} is
\[
    2^{1-n}\sum_{0\le j<\ell}\binom{n-1}{j}.
\]
Consider the unit circle ($\ell=2$). Assume that two players
choose points $P$ and $Q$, respectively. We define the dominance
relation according to the clockwise arc-length: if the clockwise
arc-length from $Q$ to $P$ is larger than the counter-clockwise one,
then we say that $Q$ dominates $P$ (or $Q$ wins); see \cite{ItMaTo}.
Obviously, there is a winner whenever all players choose points lying
on the same semicircle; otherwise, the game is in a tie. The
expected time until a winner is selected is $1/p-1$, where $p$ is the
probability that all points lie on the same semicircle. By the above
result, the latter probability is given by $p=n/2^{n-1}$ and,
consequently, the expected time is given by
\[
    n^{-1} 2^{n-1} -1.
\]
This is again a game of exponential type. Note that its discrete
version is the regular tournament Janken game with  $p_i=1/(2m+1)$
for $1\le i\le m$.

\end{document}